\documentclass[10pt,reqno]{amsart}
\usepackage[T2A]{fontenc}
\usepackage{amsmath}
\usepackage{amssymb}
\usepackage{amsfonts}
\usepackage{stmaryrd} 
\usepackage{calrsfs} 

\newtheorem{thm}{Theorem}
\newtheorem*{thm*}{Theorem}
\newtheorem{prop}{Proposition}
\newtheorem{lem}[thm]{Lemma}

\newtheorem{conj}{Conjecture}

\begin{document}

\title[The behavior of $\pi$-submaximal subgroups]{The behavior of $\pi$-submaximal subgroups under homomorphisms with  $\pi$-separable kernels}
\author{Danila O. Revin}%
\address{Danila O. Revin
\newline\hphantom{iii} Sobolev Institute of Mathematics,
\newline\hphantom{iii} 4, Koptyug av.
\newline\hphantom{iii} 630090, Novosibirsk, Russia
} \email{revin@math.nsc.ru}

\author{Andrei V. Zavarnitsine}%
\address{Andrei V. Zavarnitsine
\newline\hphantom{iii} Sobolev Institute of Mathematics,
\newline\hphantom{iii} 4, Koptyug av.
\newline\hphantom{iii} 630090, Novosibirsk, Russia
} \email{zav@math.nsc.ru}

\maketitle {\small
\begin{quote}
\noindent{{\sc Abstract.} We explore the extent to which constructing the inductive theory of $\mathfrak{X}$-submaximal subgroups is possible. To this end, we study the behavior of $\pi$-submaximal subgroups under homomorphisms with $\pi$-separable kernels and construct
examples where such behavior is irregular.}

\medskip
\noindent{{\sc Keywords:} $\mathfrak{X}$-maximal subgroup, $\mathfrak{X}$-submaximal subgroup,
complete class, $\pi$-separable group
}

\medskip
\noindent{\sc MSC2010:
20E28,  
20E34   
}
\end{quote}
}

\medskip
\hfill{\em In honor of the 50th birthday of Professor Andrey V.\,Vasil'ev}

\section{Introduction}

The present paper is concerned with the extent to which constructing the inductive theory of $\mathfrak{X}$-submaximal subgroups is possible. This concept was introduced by H.\,Wielandt during his talk at the Santa Cruz Conference on Finite Groups in \cite{80Wie} as a tool to study $\mathfrak{X}$-maximal subgroups of finite groups. We recall the background and formulate the main definitions and problems.

All groups considered throughout the text are assumed to be finite.

Since its inception in the papers by \'E.\,Galois and C.\,Jordan \cite{Galois,Jordan0,Jordan}, group theory has had the following as one of its central problems.
Given a group $G$, find its subgroups possessing a specific property or, equivalently, belonging to a specific class ${\mathfrak X}$ of groups (for example, solvable, nilpotent, abelian, $p$-groups, etc.). If ${\mathfrak X}$ has good properties resembling those of solvable groups
then to solve the general problem it suffices to know the so-called {\it maximal $\mathfrak{X}$-subgroups} (or {\em $\mathfrak{X}$-maximal subgroups}), i.\,e. the maximal by inclusion subgroup belonging to $\mathfrak{X}$.
At the famous group theory conference held in Santa Cruz as early as 1979, H.\,Wielandt
put forward the program to study the $\mathfrak{X}$-maximal subgroups in finite groups for the  so-called complete classes $\mathfrak{X}$ \cite{80Wie}. According to his definition, a nonempty class $\mathfrak{X}$ of finite groups is {\it complete} if $\mathfrak{X}$ is closed with respect to taking subgroups, homomorphic images, and extensions. The latter means that $\mathfrak{X}$ includes every group $G$ with a normal subgroup $A$ such that $A, G/A\in\mathfrak{X}$. Other than the class of solvable groups, an important example of a complete class is the class of $\pi$-groups for any set $\pi$ of primes, i.\,e. the groups all whose prime divisors of the order belong to $\pi$. Unless
otherwise stated, we will henceforth assume that $\mathfrak{X}$ is a fixed complete class. Wielandt viewed his program as a development of H\"older's program which replaced the study of an arbitrary finite group $G$ with the study of quotients of a subnormal (e.\,g., composition) series
$$
 G=G_0\geq G_1\geq\dots\geq G_n=1
$$
and, in place of a subgroup $H$ of $G$, its {\it projections}
$$H^i=(H\cap G_{i-1})G_i/G_i$$
on the corresponding quotients $G^i=G_{i-1}/G_i$ were considered. Clearly, if all projections $H^i$ are $\mathfrak{X}$-maximal in $G^i$ then $H$ itself is $\mathfrak{X}$-maximal in $G$. The problem is to see if the reverse holds, which is equivalent to  asking whether the subgroups $H\cap A$ and $HA/A$ are $\mathfrak{X}$-maximal in a normal subgroup $A$ and the quotient $G/A$ if $H$ is $\mathfrak{X}$-maximal in $G$. The answer depends on the class $\mathfrak{X}$ and determines the extent to which the problem of finding  $\mathfrak{X}$-maximal subgroups is inductive. The classes for which the answer is in the affirmative are listed in \cite{20Rev}: such is the class of all finite groups, the class of groups of order $1$, and the class of all $p$-groups for any prime $p$. For all the remaining cases, there exists a group $A$ with nonconjugate $\mathfrak{X}$-maximal subgroups. Wielandt proposed \cite[4.3]{80Wie} a rather general construction showing that, whichever finite group $G$ is, there exists an epimorphism from the regular wreath product $G^*=A\wr G$ onto $G$ such that every (not only maximal!) $\mathfrak{X}$-subgroup of $G$ is the image of an $\mathfrak{X}$-maximal subgroup of~$G$.

Wielandt's idea \cite{80Wie} was to consider instead of $\mathfrak{X}$-maximal subgroups the broader, yet substantial, concept of $\mathfrak{X}$-submaximal subgroups which ``behave well''
under intersections with normal and subnormal subgroups. A subgroup $H$ of $G$ is
{\it $\mathfrak{X}$-submaximal} (or a {\it submaximal $\mathfrak{X}$-subgroup}) if there exists an embedding of $G$ as a subnormal subgroup into a group $G^*$ (in which case we say that $G$ is {\em subnormally embedded} in $G^*$) such that $H$ coincides with the intersection of $G$ and $K$ for a suitable $\mathfrak{X}$-maximal subgroup $K$ of $G^*$. Clearly, every $\mathfrak{X}$-maximal subgroup is also $\mathfrak{X}$-submaximal. The converse is not true, see \cite[p.~13]{18GuoRev1}. $\mathfrak{X}$-submaximal subgroups have the obvious inductive property: if $H$ is $\mathfrak{X}$-maximal and $A$ is subnormal (in particular, normal) in~$G$ then $H\cap A$ is $\mathfrak{X}$-submaximal in~$A$. Not every $\mathfrak{X}$-subgroup is submaximal. An obstruction here is the following

\begin{thm*}[Wielandt--Hartley, {\cite[5.4(a)]{80Wie}}, {\cite[Theorem~2]{19RevSkrV}}]
If $H$~is an $\mathfrak{X}$-sub\-maxi\-mal subgroup of $G$ then $N_G(H)/H$ includes no nontrivial $\mathfrak{X}$-subgroups.
\end{thm*}

It is due to this theorem that the concept of an $\mathfrak{X}$-submaximal subgroup becomes useful and efficient. For example, it helps to easily see that the $\mathfrak{X}$-maximal subgroups are determined uniquely up to conjugacy by their projections on the quotients of a subnormal series \cite[5.4(c)]{80Wie}, \cite[Corollary~1]{19RevSkrV}. In comparison, the similarly defined dual concept of
an $\mathfrak{X}$-epimaximal subgroup which ensures ``good behaviour'' under homomorphisms
turns out to lack content, because in any nontrivial situation it is equivalent to the concept of an $\mathfrak{X}$-subgroup \cite{20Rev}. It was shown in \cite{18GuoRev1,18GuoRev2} that the knowledge of $\mathfrak{X}$-submaximal subgroups in simple groups for a given class $\mathfrak{X}$ would make it possible to inductively construct the $\mathfrak{X}$-maximal subgroups in an arbitrary finite group and, consequently, would make great progress in solving the general problem.

One can naturally ask to what extent the problem of finding  $\mathfrak{X}$-submaximal subgroups  is inductive. By this, we mean the following. Besides the fact that $\mathfrak{X}$-submaximal subgroups ``behave well'' with respect to intersections with (sub)normal subgroups, it is important to determine how they behave under homomorphisms. Due to the above, we cannot expect the image of an  $\mathfrak{X}$-submaximal subgroup in a factor group to be $\mathfrak{X}$-submaximal: this fails even for the images of $\mathfrak{X}$-maximal subgroups.
At the same time, generalizing a famous result by S.\,A.\,Chunikhin (see \cite[Ch. 5, Theorem 3.7]{86Suz})  Wielandt showed
that the following reduction theorem holds:

\begin{thm*}[Wielandt, {\cite[12.9]{94Wie607}}] If $A$ is an $\mathfrak{X}$-separable normal subgroup of a group~$G$ (i.e., $A$ has a subnormal series with every factor either belonging to $\mathfrak{X}$ or to the class $\mathfrak{X}'$ that consists of all groups none of whose nonidentity subgroups belongs to~$\mathfrak{X}$), then the canonical epimorphism $G\rightarrow G/A$ maps an $\mathfrak{X}$-maximal subgroup to an $\mathfrak{X}$-maximal one and, furthermore, induces a bijection between the conjugacy classes of $\mathfrak{X}$-maximal subgroups in $G$ and $G/A$.
\end{thm*}

Therefore, the presence of a normal $\mathfrak{X}$-separable subgroup $A$ in $G$ enables us to replace the study of $\mathfrak{X}$-maximal subgroups in $G$ with a similar problem for the smaller group $G/A$.
The inductiveness of the problem of finding $\mathfrak{X}$-subnormal subgroups depends, to a large extent, on whether a similar reduction theorem for $\mathfrak{X}$-submaximal subgroups holds, i.\,e. whether the following is true.

\begin{conj}\label{conj_a}
 If $A$ is an $\mathfrak{X}$-separable normal subgroup of a finite group $G$ then the canonical epimorphism $G\rightarrow G/A$ always maps an $\mathfrak{X}$-submaximal subgroup of $G$ to an $\mathfrak{X}$-submaximal subgroup of $G/A$ and vice versa, every $\mathfrak{X}$-submaximal subgroup of $G/A$ is the image of an $\mathfrak{X}$-submaximal subgroup of $G$.
\end{conj}

This conjecture takes central place in the present paper. The original idea of the authors was to attempt confirming both Conjecture \ref{conj_a} and

\begin{conj}\label{conj_b}
 If $A$ is a normal and $H$ an $\mathfrak{X}$-submaximal subgroups of a finite group $G$ then $H$ is $\mathfrak{X}$-submaximal in $N_G(H\cap A)$.
\end{conj}

In relation to Conjecture \ref{conj_b}, we note that if $H$ is an $\mathfrak{X}$-submaximal subgroup of $G$ and $H$ is contained in a subgroup $M$ of $G$ then $H$ is not necessarily  $\mathfrak{X}$-submaximal in $M$. Say, an $\mathfrak{X}$-submaximal but not  $\mathfrak{X}$-maximal subgroup $H$ which is contained in a strictly larger  $\mathfrak{X}$-subgroup $M$ is not $\mathfrak{X}$-submaximal in $M$.

If both Conjectures \ref{conj_a} and \ref{conj_b} turned out to be true, the problem of finding $\mathfrak{X}$-submaximal subgroups in an arbitrary finite group could be considered inductive. Let us explain why. Suppose $G$ is not simple and we know how to find the $\mathfrak{X}$-submaximal subgroups in all groups whose order is smaller than the order of $G$. Let~$A$ be a proper normal subgroup of $G$. For any $\mathfrak{X}$-submaximal subgroup $H$ of $G$, the intersection $H\cap A$ is $\mathfrak{X}$-submaximal in $A$ and is therefore known. We may consider the subgroup $N_G(H\cap A)$ in which the $\mathfrak{X}$-subgroup $H$ is submaximal
due to Conjecture B. Furthermore, $N_G(H\cap A)$ includes the normal subgroup $N_A(H\cap A)$ which is $\mathfrak{X}$-separable by the Wielandt--Hartley theorem. Thus, the validity of Conjecture A would reduce finding the $\mathfrak{X}$-submaximal subgroups in $N_G(H\cap A)$ to a similar  problem for $N_G(H\cap A)/N_A(H\cap A)\simeq AN_G(H\cap A)/A$. The quotient $G/A$ acts naturally  on the set of conjugacy classes of $\mathfrak{X}$-submaximal subgroups of~$A$, and $AN_G(H\cap A)/A$ is the stabilizer of the class containing $H\cap A$. Therefore, the problem of finding $\mathfrak{X}$-submaximal subgroups of $G$ is reduced to a similar problem for the stabilizers in $G/A$ of the conjugacy classes of $\mathfrak{X}$-submaximal subgroups of~$A$. We note that, for a simple group $S$, the $\mathfrak{X}$-submaximal subgroups are precisely the intersections of $S=\operatorname{Inn}(S)$ and $\mathfrak{X}$-maximal subgroups of $\operatorname{Aut}(S)$ \cite[5.3]{80Wie}, \cite[Proposition~7]{18GuoRev}.

The fact that Conjecture \ref{conj_a} does not hold was established in \cite{20RevZav}, where an example was constructed of a group $G$ with a normal abelian $2$-subgroup $A$ and a submaximal but not maximal $\{2,3\}$-subgroup $H$ whose image with respect to the canonical epimorphism $G\rightarrow G/A$ is not $\{2,3\}$-submaximal in~$G/A$. The kernel of this homomorphism is an $\mathfrak{X}$-subgroup for the class $\mathfrak{X}$ of all $\{2,3\}$-groups. In the same paper, it was announced that there exist epimorphisms whose kernel belongs to $\mathfrak{X}'$ and the image of an $\mathfrak{X}$-submaximal subgroup is not $\mathfrak{X}$-submaximal. We construct an infinite series of such examples in the final section of this paper. The situation with the preimages of $\mathfrak{X}$-submaximal subgroups is completely analogous, which makes Conjecture~\ref{conj_a} invalid in both directions.

When constructing the examples, we consider a group $G$ with a unique minimal normal subgroup $V$ which is abelian, is not contained in the center of $G$, and is such that $G/V$ is a nonabelian simple group. In order to justify the examples, we need to show that some $\pi$-subgroup of $G$ is not submaximal. We make use of Proposition~\ref{one} stated and proved in Section~\ref{Sec3}, which implies that the $\pi$-submaximal subgroups of $G$ are exhausted by the intersections of $G$ with $\pi$-maximal subgroups of the groups $G^*$ satisfying $G=\operatorname{Inn}(G)\trianglelefteqslant G^* \leqslant\operatorname{Aut}(G)$. We emphasize that the subnormal embedding of $G$ into $G^*$ in the definition of an $\mathfrak{X}$-submaximal subgroup cannot in general be substituted with a normal embedding. An appropriate series of examples was pointed out by A.V.\,Vasil'ev, see \cite[Section~2]{19RevSkrV}. Thus, Proposition~\ref{one} is of interest in its own right and can be considered as one of the main results of this paper.

Although the problem of whether Conjecture B is true loses its relevance in light of the present results along with~\cite{20RevZav}, it still remains open.

Undoubtedly, the examples constructed in~\cite{20RevZav} and in this paper by no means imply that the study of $\mathfrak{X}$-submaximal subgroups is futile. Instead, they lead to the realization as to why the central problem in Wielandt's program and in the topic related to the search of $\mathfrak{X}$-maximal subgroups for a given complete class~$\mathfrak{X}$ is the description of $\mathfrak{X}$-submaximal subgroups in simple groups or, equivalently, the description of  $\mathfrak{X}$-maximal subgroups in almost simple groups. Such a description would make it possible to find the $\mathfrak{X}$-maximal subgroup in arbitrary finite groups.

\section{Preliminaries}

To simplify the exposition, we restrict ourselves to the case where $\mathfrak{X}$ is the class of $\pi$-groups for an arbitrary set $\pi$ of primes.

The following lemma was first formulated in \cite[Statement 5.4(a)]{80Wie}. A proof can be found in \cite[Theorem 2]{19RevSkrV}.

\begin{lem}[The Wielandt--Hartley theorem, strong form]\label{WH-strong}
  Let $A$ be a subnormal and $K$ a $\pi$-maximal subgroup of a finite group $G$. Then $N_A(K\cap A)/(K\cap A)$
  is a $\pi'$-group.
\end{lem}

\begin{lem}[{\cite[Theorems 12.4 and 12.7]{94Wie607}}]\label{pimhom} If $U$ is a normal $\pi$- or $\pi'$-subgroup of a group $G$ then $KU/U$ is $\pi$-maximal in $G/U$
  for every $\pi$-maximal subgroup $K$ of $G$
\end{lem}

\begin{lem}\label{three}
  Let $G\trianglelefteqslant \trianglelefteqslant G^*$ and let $H=K\cap G$ for some $\pi$-maximal subgroup $K$ of $G^*$.
  Suppose that $U$ is a normal $\pi'$-subgroup of $G^*$ and $\,\overline{\phantom{a}}:G^*\to G^*/U$ is the canonical epimorphism.
  Then $\overline{H}=\overline{K}\cap\overline{G}$.
\end{lem}
\begin{proof} Clearly, $\overline{H}=\overline{K\cap G}\leqslant \overline{K}\cap \overline{G}$. Suppose that
$\overline{H}< \overline{K}\cap \overline{G}$. Since $G\trianglelefteqslant \trianglelefteqslant G^*$,
we have $H\trianglelefteqslant \trianglelefteqslant K$, $\overline{H} \trianglelefteqslant \trianglelefteqslant \overline{K}$,
and $\overline{H} \trianglelefteqslant \trianglelefteqslant \overline{K}\cap \overline{G}$.

Since $\overline{H}<\overline{K}\cap \overline{G}$, the index of $\overline{H}$ in $N_{\overline{G}}(\overline{H})$ is divisible by a prime
$p\in \pi$. We have
$$
N_{\overline{G}}(\overline{H})=N_G(H)U/U=\overline{N_G(H)}.
$$
Indeed, we clearly have  $\overline{N_G(H)}\leqslant N_{\overline{G}}(\overline{H})$.
Suppose $x\in N_G(HU)$ (equivalently, $\overline{x}\in N_{\overline{G}}(\overline{H})$). Then $H^xU=HU$ and the Schur--Zassenhaus theorem
implies that there is $u\in U$ such that $H^x=H^u$, i.\,e. $xu^{-1}\in N_G(H)$,  $x\in N_G(H)U$, and $\overline{x}\in\overline{N_G(H)}$.
Now, we have
$$
p\mid \left|N_{\overline{G}}(\overline{H}):\overline{H}\right|=\left|\frac{N_G(H)U}{HU}\right|=\frac{|N_G(H)|}{|N_U(H)|}:\frac{|H|}{|H\cap U|}=\left|\frac{N_G(H)}{H}\right|:\frac{|N_U(H)|}{|H\cap U|}.
$$
Since $|N_U(H)|$ is a $\pi'$-number, we have $p\mid \left|N_G(H)/H\right|$ contrary to Lemma \ref{WH-strong}.
\end{proof}

Let $\mathcal{X}$ be a linear representation of a group $G$. Given $\gamma\in \operatorname{Aut}(G)$, the conjugate representation $\mathcal{X}^\gamma$ is defined by
$\mathcal{X}^\gamma(g^\gamma)=\mathcal{X}(g)$ for all $g\in G$.

\begin{lem}\label{norim} Let $G$ be a finite group such that $Z(G)=1$. Let $\mathcal{X}:G\to \operatorname{GL}_n(F)$ be a faithful absolutely irreducible representation of $G$ over a field $F$. Suppose that $\mathcal{X}$ is not equivalent to $\mathcal{X}^\mathcal{\gamma}$ for every $\gamma\in \operatorname{Aut}(G)\setminus\operatorname{Inn}(G)$. Then $N_{\operatorname{GL}_n(F)}(\operatorname{Im}\mathcal{X})\cong C\times G$ where $C\cong F^\times$ is the scalar subgroup of $\operatorname{GL}_n(F)$.
\end{lem}

\begin{proof} Denote $M=\operatorname{Im}\mathcal{X}\cong G$ and $N=N_{\operatorname{GL}_n(F)}(M)$. Since $\mathcal{X}$ is absolutely irreducible, we have
$C=C_{\operatorname{GL}_n(F)}(M)$ by Schur's lemma. Clearly, $C\trianglelefteqslant N$ and
$C\cap M =1$ due to $Z(G)=1$.
Therefore,  $G\leqslant \overline{N}\leqslant \operatorname{Aut}(G)$, where we
denote by $\overline{\phantom{m}}: N\to N/C$ the canonical epimorphism.
It remains to see that $G=\overline{N}$.

Assume to the contrary that there is $\gamma\in \operatorname{Aut}(G)\setminus\operatorname{Inn}(G)$ such that $\overline{t}=\gamma$ for some $t\in N$.
The identification of $G$ and $\operatorname{Inn}(G)$ shows that
$\mathcal{X}$ and $\overline{\phantom{m}}$ are mutually inverse isomorphisms between $M$ and $G$.
Namely, $x=\mathcal{X}(\overline{x})$ for every $x\in M$ and $\overline{\mathcal{X}(g)}=g$ for every $g\in L$. Since $t$ normalizes $M$, we have $\mathcal{X}(g)^t\in M$ and so
$$
\mathcal{X}(g)^t=\mathcal{X}(\overline{\mathcal{X}(g)^t})=
\mathcal{X}(\overline{\mathcal{X}(g)}^{\,\overline{t}})=\mathcal{X}(g^\gamma)=\mathcal{X}^{\gamma^{-1}}(g)
$$
which implies that the conjugate representation $\mathcal{X}^{\gamma^{-1}}$
is equivalent to $\mathcal{X}$. This contradicts the hypothesis, because $\gamma^{-1}\in\operatorname{Aut}(G)\setminus\operatorname{Inn}(G)$. The claim follows.
\end{proof}

\section{Proposition}\label{Sec3}

The following result includes \cite[Proposition~1]{20RevZav} as a particular case.

\begin{prop}\label{one}
  Suppose a finite group $G$ has a unique minimal normal subgroup~$V$ which is abelian and $V\nleqslant \operatorname{Z}(G)$.
  Suppose also that $L=G/V$ is a nonabelian simple group.  Let $H$ be a $\pi$-submaximal subgroup of $G$ and let $G^*$ have
  minimal order among the groups such that  $G\trianglelefteqslant \trianglelefteqslant G^*$ and
  $H=K\cap G$ for a $\pi$-maximal subgroup $K$ of $G^*$. Then $G\trianglelefteqslant G^*$ and $C_{G^*}(G)=1$.
\end{prop}
\begin{proof}
Since $V$ is abelian minimal normal, it is an elementary abelian $p$-group for a prime $p$.
Denote $W=\langle V^g\mid g \in G^*\rangle$, the normal closure of $V$ in $G^*$. Note that $W$ is a $p$-group as the subgroup
generated by subnormal $p$-subgroups. If $p\in \pi$ then $W\leqslant K$, and if $p\not\in\pi$ then $W\cap K= 1$.
Denote by $\overline{\phantom{a}}:G^*\to G^*/W$ the canonical epimorphism. Set $X=\langle G^g\mid g \in G^*\rangle$.
The minimality of $|G^*|$ implies that $G^*=KX$. Moreover, $\overline{X}$ is minimal normal in $\overline{G^*}$.
In particular, $GW\trianglelefteqslant X$.

We show that every minimal normal subgroup $U$ of $G^*$ such that $U\nleqslant W$ is a $\pi$-group. Indeed, we have
$U\cap W=1$ and $\overline{U}\cong U$. Also, $\overline{U}$ is minimal normal in $\overline{G^*}$, for
if $\overline{M}\trianglelefteqslant \overline{G^*}$ and $M\leqslant U$ then
$$
[M,G^*]\leqslant U\cap MW = (U\cap W)M=M
$$
and $M\trianglelefteqslant G^*$. Therefore, either $\overline{U}=\overline{X}$ or $\overline{U}\cap \overline{X}=1$. In the former case,
we have $X=UW\cong U\times W$, which is impossible in view of the structure of $G$. In the latter case, $U\cap X\leqslant U\cap W = 1$
and $U$ can be embedded into $G^*/X\cong K/(K\cap X)$, i.\,e. $U$ is a $\pi$-group.

Consequently, every minimal normal subgroup $U$ of $G^*$ is either a $\pi$- or a $\pi'$-group, for we have either $U\nleqslant W$ or
$U\leqslant W$, and $W$ is a $\pi$- or a $\pi'$-group depending on whether $p\in \pi$.

We show that $W$ is a unique minimal normal subgroup of $G^*$. Let $U$ be a minimal normal subgroup of $G^*$. It suffices to show that
$U\cap V\ne 1$ and so $V\leqslant U$ and $W=\langle V^g\mid g \in G^*\rangle\leqslant U$. Assume to the contrary that $U\cap V= 1$.

Denote by $\widetilde{\phantom{a}}:G^*\to G^*/U$ the canonical epimorphism. By Lemma \ref{pimhom}, $\widetilde{K}$ is a $\pi$-maximal
subgroup of $\widetilde{G^*}$. We have $G\cap U\leqslant V\cap U=1$ as $V$ is the unique minimal normal subgroup of $G$.
It follows that $G\cap U=1$ and $G\cong \widetilde{G}$. We show that $\widetilde{H}=\widetilde{K}\cap \widetilde{G}$.
This follows from Lemma \ref{three} if $U$ is a $\pi'$-group. If $U$ is a $\pi$-group, we have $U\leqslant K$ and $GU\cap K=(G\cap K)U=HU$.
This again implies $\widetilde{H}=\widetilde{K}\cap \widetilde{G}$. We now have a contradiction with the minimality of $|G^*|$, and so
$U\cap V\ne 1$ as claimed.

Let us now show that $W =VC_W(G)$. By Clifford's theorem and the subnormality of $G$ in $G^*$, the $\mathbb{F}_pG$-module $W$ is
completely reducible. An arbitrary irreducible submodule $U$ of $W$ that is not contained in $C_W(G)$ must be contained in $V$. This
follows since $W$ normalizes $G$, see \cite[Theorem 2.6]{08Isa}, and
$$
U=[U,G]\leqslant U\cap G\leqslant W\cap G=V.
$$
Therefore, $W =VC_W(G)$.

We now prove that $C_{G^*}(X)=1$. Assuming the contrary we have $W\leqslant C_{G^*}(X)\trianglelefteqslant G$, because $W$ is the unique
minimal normal subgroup of $G^*$. Then $V\leqslant W\cap G\leqslant C_G(G)=\operatorname{Z}(G)$ which contradicts the assumption that
$V\nleqslant \operatorname{Z}(G)$.

Denote $N=N_K(GW)$, $G^0=N_{G^*}(GW)$. We have $G^0=NX$, since $X\leqslant G^0$ and $G^*=KX$.

We show that if $M$ is a $\pi$-maximal subgroup of $G^0$ that contains $N$ then $H=G\cap M$. Firstly, we have $H=G\cap N$,
because $H\leqslant N$ and $G\cap N\leqslant G\cap K=H$. Secondly, $G^0=NX=MX$ by the above.

Let $1=g_1,\ldots,g_m$ be a right transversal of $N$ in $K$ which will also be a right transversal of $G^0$ in $G^*$.
We set $M_i=(M\cap GW)^{g_i}$. For every $g\in K$, there exists $\sigma\in \operatorname{Sym}_m$ and $t_1,\ldots,t_m\in N$ such that
$g_ig=t_ig_{i\sigma}$. Therefore,
$$
M_i^g=(M\cap GW)^{g_ig}=(M\cap GW)^{t_ig_{i\sigma}}=(M\cap GW)^{g_{i\sigma}}=M_{i\sigma}.
$$
It follows that $K$ normalizes the subgroup $M_X=\langle M_i\mid i=1,\ldots, m\rangle$. Hence, $\overline{K}$ normalizes $\overline{M_X}$.
Note that $\overline{M_X}$ is a $\pi$-group, since $[\overline{M_i},\overline{M_j}]=1$ for $i\ne j$ and the $M_i$'s are $\pi$-groups.
Also, $\overline{M_X}\leqslant \overline{K}$, because $\overline{K}$ is $\pi$-maximal in $\overline{G^*}$.

Now, if $W$ is a $\pi$-group, we have $M_X\leqslant K$ and
$$
M\cap G\leqslant M\cap GW=M_1 \leqslant M_X\leqslant K,
$$
which yields
$$
H=N\cap G\leqslant M\cap G\leqslant K\cap G =H
$$
as claimed. Suppose $W$ is a $\pi'$-group. Then Lemma \ref{pimhom} implies
$$
\overline{K}\cap\overline{G}=\overline{H}\leqslant \overline{N}\cap\overline{G}\leqslant \overline{M}\cap \overline{G}\leqslant
\overline{M_X}\cap\overline{G}\leqslant\overline{K}\cap\overline{G},
$$
which yields $\overline{H}=\overline{M}\cap\overline{G}$. Another application of Lemma \ref{pimhom}
gives $\overline{M}\cap\overline{G}=\overline{M\cap G}$, hence $HW=(M\cap G)W$. Therefore, $|H|=|M\cap G|$. Since
$H=N\cap G\leqslant M\cap G$, we have $H=M\cap G$ as claimed.

The minimality of $G^*$ now gives $X=GW$ and $X/W\cong L$. Therefore, $C_W(G)=C_W(X)\leqslant C_{G^*}(X)=1$ and
$W=VC_W(G)=V$. Consequently, $X=GV=G$ and $G\trianglelefteqslant G^*$.
\end{proof}

\section{Examples}
In this section, we construct an infinite series of examples where the image of a $\pi$-submaximal subgroup under an epimorphism $\phi$ whose kernel is an abelian $\pi'$-group is not $\pi$-submaximal in $\operatorname{Im}\phi$. Conversely, we give examples where a $\pi$-submaximal subgroup of $\operatorname{Im}\phi$ is the image of no $\pi$-submaximal in the domain of $\phi$.
In these examples, $\pi=\{2,3\}$.

The simple group $L=\operatorname{PSL}_2(7)$ has presentation
\begin{equation}\label{pres}
L=\langle a,b \mid a^2=b^3=(ab)^7=[a,b]^4=1 \rangle,
\end{equation}
see \cite{atl}. Let $F$ be a finite field of characteristic coprime with $|L|=2^3.3.7$ in which the polynomial $x^2+x+2$ has distinct roots, say, $\alpha$ and $\beta$. For example, such is every field $\mathbb{F}_p$ for an odd prime $p\equiv 1,2,4 \pmod{7}$.
It can be readily seen that the matrices
$$
A=\left(
    \begin{array}{ccc}
      1& \alpha & \beta \\
      0 & -1 &  0 \\
      0 &  0 & -1 \\
    \end{array}
  \right),\qquad
B=\left(
    \begin{array}{ccc}
      0 & 1 & 0 \\
      0 & 0 & 1 \\
      1 & 0 & 0 \\
    \end{array}
  \right)
$$
satisfy the defining relations in (\ref{pres}) so that the map $a\mapsto A$, $b\mapsto B$ determines an absolutely irreducible faithful representation $\mathcal{X}:L\to \operatorname{GL}_3(F)$ whose $F$-character $\chi$ is shown in Table \ref{bch}.

\begin{table}[htb]
\centering
\caption{The $F$-character of $\mathcal{X}$}\label{bch}
\begin{tabular}{c|cccccc}
                         & $1a$ & $2a$ & $3a$ & $4a$ & $7a$ & $7b$ \\ [2pt]
   \hline
 $\chi^{\vphantom{A^{A^A}}}$ & $3$  & $-1$ & $0$ & $1$ & $\alpha$ & $\beta$
 \end{tabular}
\end{table}

We identify $L$ with the subgroup $\operatorname{Inn}(L)$ of $\operatorname{Aut}(L)\cong \operatorname{PGL}_2(7)$. It is known that $\operatorname{Aut}(L)$ is the extension $L\langle \delta \rangle$, where $\delta$ is a diagonal automorphism of $L$ of order~$2$.

Observe that the conjugate representation $\mathcal{X}^\delta$ is not equivalent to $\mathcal{X}$, because $\delta$ permutes the two conjugacy classes of $L$ of order $7$ while $\chi$ has distinct values on these classes.

Let $V$ be the $3$-dimensional $FL$-module corresponding to $\mathcal{X}$ and let $G$ be the natural semidirect product of $V$ and $L$. Note that $Z(G)=1$ and so we may identify $G$ with $\operatorname{Inn}(G)$ inside $\operatorname{Aut}(G)$.

\begin{lem}\label{autg}
$\operatorname{Aut}(G)/V \cong C\times L$, where $C\cong F^\times$.
\end{lem}
\begin{proof} Observe that $V$ is characteristic in $G$ and so every automorphism of $G$ leaves~$V$ invariant. Hence, we may apply the general theory of automorphisms of group extensions. By \cite[(4.4)]{82Rob}, there is an exact sequence of groups
$$
0\to Z^1(L,V) \to \operatorname{Aut}(G) \to N_{\operatorname{GL}_3(F)}(\operatorname{Im}\mathcal X)
\to H^2(L,V).
$$
Since the characteristic of $F$ is coprime to $|L|$, we have $H^2(L,V)=H^1(L,V)=0$.
Also, $H^1(L,V)=Z^1(L,V)/B^1(L,V)$ and
$B^1(L,V)\cong V/C_V(L)$. Therefore, we have $Z^1(L,V)=B^1(L,V)\cong V$ due to $C_V(L)=0$.

As we observed above, $\mathcal{X}^\delta$ is not equivalent to $\mathcal{X}$. Also, $Z(L)=1$, since $L$ is simple. Hence, $N_{\operatorname{GL}_3(F)}(\operatorname{Im}\mathcal X)\cong C\times L$ by Lemma \ref{norim}. The claim follows from these remarks.
\end{proof}

Let $\pi=\{2,3\}$.
As was observed in \cite[Example 2, p. 170]{86Suz}, the $2$-Sylow subgroups of $L$
are not $\pi$-maximal, because $L$ includes subgroups of order $2^3\cdot3$, however
they are $\pi$-submaximal, because they are intersections of $L$ with the $2$-Sylow subgroups
of $\operatorname{Aut}(L)$ which are maximal.

\begin{lem}\label{s2n}
  In the above notation, the $2$-Sylow subgroups of $G$ are not $\pi$-sub\-ma\-xi\-mal.
\end{lem}
\begin{proof}
Let $S$ be a $2$-Sylow subgroup of $G$ and suppose to the contrary that $S$ is $\pi$-submaximal.
Let $G^*$ be a finite group of minimal order such that
$G\trianglelefteqslant \trianglelefteqslant G^*$ and there is a $\pi$-maximal subgroup $K$ of $G^*$ satisfying $S=G\cap K$. Proposition \ref{one} implies that $G\trianglelefteqslant G^*$ and $C_{G^*}(G)=1$, i.\,e. $G^*\leqslant \operatorname{Aut}(G)$. Let
$\overline{\phantom{m}}: \operatorname{Aut}(G)\to \operatorname{Aut}(G)/V$ be the canonical epimorphism.
Since $V$ is a $\pi'$-group, we have $\overline{S}=\overline{G}\cap \overline{K}$ by Lemma~\ref{three}, $\overline{S}$ is a $2$-Sylow subgroup of $\overline{G}\cong L$, and $\overline{K}$ is $\pi$-maximal in $\overline{G^*}$ by Lemma~\ref{pimhom}.

By Lemma \ref{autg}, we have $\overline{G^*}\cong C_0 \times L$, where $C_0$
is a subgroup of $F^\times$. In particular, $K$
is the direct product of a $\pi$-maximal subgroup of $C_0$ and a $\pi$-maximal subgroup of $\overline{G}\cong L$. Thus, $\overline{S}$ is $\pi$-maximal in $L$ contrary to the observation above. This completes the proof.
\end{proof}

The following example shows that the homomorphic image of a $\pi$-submaximal subgroup is not necessarily $\pi$-submaximal.

\medskip

{\bf Example 1.} Let $V^*$ denote the $FL$-module contragredient to $V$. Since $\delta$ interchanges $V$ and $V^*$, it naturally acts on $V\oplus V^*$. Hence, the semidirect product $H=(V\oplus V^*)L$ can be extended to $H^*=H\langle\delta\rangle$. As above, let $\pi=\{2,3\}$. The $2$-Sylow subgroup of $H^*$ is $\pi$-maximal, because its homomorphic image in $L\langle\delta\rangle$ is $2$-Sylow which is a maximal subgroup. Thus, the $2$-Sylow subgroup $S$ of $H$ is $\pi$-submaximal. Let $\overline{\phantom{m}}: H\to H/V^*$ be the canonical epimorphism.
The image $\overline{H}$ is isomorphic to the semidirect product $G=VL$, and
Lemma \ref{s2n} implies that $\overline{S}$ is not $\pi$-submaximal.

\medskip

A $\pi$-submaximal subgroup of a homomorphic image need not be the homomorphic image of a $\pi$-submaximal subgroup as shows the next example.

\medskip

{\bf Example 2.} Consider the canonical epimorphism  $\overline{\phantom{m}}: G\to G/V\cong L$.
We noted above that a $2$-Sylow subgroup $S$ of $L$ is $\pi$-submaximal. However, there is no $\pi$-submaximal subgroup $T$ of $G$ such that $\overline{T}=S$, because $T$ would clearly need to be $2$-Sylow in $G$, but the $2$-Sylow subgroups of $G$ are not $\pi$-submaximal by Lemma~\ref{s2n}.

\medskip

In connection with the study of $\pi$-submaximal subgroups in minimal nonsolvable groups begun in
\cite{18GuoRev}, the paper \cite{20RevZav} gives an example of a minimal nonsolvable group
$G$ such that the Frattini subgroup $\Phi(G)$ is a $\pi$-group, where $\pi=\{2,3\}$, and a
$\pi$-submaximal subgroup of the minimal simple group $G/\Phi(G)\cong \operatorname{PSL}_2(7)$ is not the homomorphic image of any $\pi$-submaximal subgroup of $G$. Furthermore, it was shown in \cite{18GuoRev} that, for every   minimal nonsolvable group $G$, the image of a $\pi$-submaximal subgroup in $G/\Phi(G)$ is always $\pi$-submaximal in the minimal simple group $G/\Phi(G)$. In the same paper, the $\pi$-submaximal subgroups in minimal simple groups were classified. It would be interesting to see if there exists  a minimal nonsolvable group $G$ such that $\Phi(G)$ is a $\pi'$-group and $G/\Phi(G)$ has a $\pi$-submaximal subgroup that is not the image of any $\pi$-submaximal subgroup of $G$.

\medskip

{\em Acknowledgment.} This work was funded by RFBR and BRFBR, project \textnumero\ 20-51-00007 and by the Program of Fundamental Scientific Research of the SB RAS \textnumero\ I.1.1., project \textnumero\ 0314-2016-0001.

\end{document}